\documentclass[12pt,reqno]{amsart}
\usepackage{amsmath,amsfonts,amssymb,amsthm,hyperref,enumerate,multicol}
\usepackage{booktabs}
\usepackage{graphicx}
\usepackage{tikz}
\usetikzlibrary{shapes.geometric, arrows}
\newtheorem{lemma}{Lemma}
\newtheorem{theorem}{Theorem}

\newtheorem{example}{Example}

\numberwithin{equation}{section}
\begin{document}
\dedicatory{This work is dedicated to the evolving field of approximation theory, with hope it connects the dot between mathematics and real-world phenomena}



\title[]
{On a novel probabilistic Sampling Kantorovich operators and their application}
\maketitle

\begin{center}
{\bf Digvijay Singh$^1$ Rahul Shukla$^2$, Karunesh Kumar Singh$^3$ } \footnote{Corresponding author: Karunesh Kumar Singh} \\
\vskip0.15in
$^{1,3}$ Department of Applied Sciences and Humanities, Institute of Engineering and Technology, Lucknow, 226021, Uttar Pradesh, India \\
$^{2}$ Department of Mathematics, Deshbandhu College, University of Delhi, India \\
\vskip0.15in

\vskip0.15in

Email: dsiet.singh@gmail.com$^1$, rshukla@db.du.ac.in$^2$,   kksiitr.singh@gmail.com$^3$,
\end{center}

\begin{abstract}
This article starts with the fundamental theory of stochastic type convergence and the significance of uniform integrability in the context of expectation value. A novel probabilistic sampling kantorovich (PSK-operators) is established with the help of classical sampling operators (SK-operators). We establish the proof of the fundamental theorem of approximation and a lemma corresponding to the PSK- operators. Moreover, some examples are illustrated not only in numerical form but also in a detailed study of some important features of an image at different samples. Eventually, a comparative analysis is made on the basis of some parameters like peak signal noise ratio (PSNR), structural similarity index (SSIM) etc. between the classical and probabilistic sense in tabulated form, which connects the whole dots of the theory present in the article.
\end{abstract}


\section{Background}
The summability theory originates from the sequence space, comes from the branch of pure mathematics, in which the properties regarding convergence and divergence are studied. Mathematician of this era have been interested in the structural theory of topological vector spaces,law of large numbers (LLN), summability theory and combined study of LLN, summmability and integrability. Consequently, because of a great scope of research not only in theoretical and practical application, the classical results of summability theory pushes us to study the sequence spaces. Mainly, the vital role of summability theory is to make the non-convergence series to convergent. Moreover, the sequence space gives leverage to study local properties of summability such as the \`{C}es\`{a}ro matrix etc. \par
The modern theory of summability starts with the matrix transformation which takes one sequence space to another sequence space. As a result, in 1950, Abraham Robinson \cite{P1} gave an idea to study the infinite real or complex number in the form of linear operators over Banach spaces. In 1999,  P. V. Subrahmanyam \cite{P2} established another version of summability in the form of fuzzy numbers. After that, many researchers \cite{P5, P3,P4} have contributed in the area of fuzzy summability in terms of \textit{P}- summability, \textit{A}-statistical convergence in fuzzy numbers, fuzzy Korovkin type. The core application of summability theory comes from the area like probability theory, approximation theory, and differential equations, as shown in \cite{S3,S9,S10,S6,S4,S1,S2,S5,S8,S7}. Recently, Demirci.et.al give a new type statistical summability on scaling time \cite{Turkey1} and Dinar. at.al. \cite{Turkey2} put light on the statistical for of operators via power series summability.\par 
 In past two decades, one of the emerging field of approximation theory comes from the field of sampling theory. Some great researchers contribute significantly in the field of sampling theory and tackle many real-world problems using approximation of operators. For more detail, we suggest the keen readers to go through the articles \cite{RG,ss5,ss6,ss3,ss2,ss1,ss4}. Motivated from these researches, Bardaro with his colleague intuited a novel sampling type operators given as follows;  
\begin{eqnarray}\label{DSK}
    \mathfrak{S}_n(f)(x) := n^d \sum_{k \in \mathbb{Z}^d} \xi(n x - k) \int_{R_k^n} f(u) \, du,
\end{eqnarray}
where the meaning of the parameters $\xi_1,\, \&\,\xi_2$ ar given in section \eqref{Sec2}\par
 Moreover, it has been seen \cite{S11} that the study of stochastic type convergence remains incomplete until we study the stochastic type of uniform integrability, as it provides a condition that bridges the gap between the convergence of random variables and their convergence of expectation of random variables. \par
  Now, keeping all these facts in mind, we do some sort of manipulation as; \( f \in L^1(\mathbb{R}^q) \) is replaced a function with noisy measurement and the integrals are affected by stochastic perturbation, assuming
\( \mathfrak{P}_n(f)(x) \) a random variable. Moreover, while proving the fundamental theorem of approximation \eqref{TH1}, we analyze convergence in expectation, probability, or almost surely defined in \eqref{M0}. Let the observed data be affected by random noise:
\[
f_\beth(x) = f(x) + \varepsilon_\beth(x),
\]
where \( \varepsilon_\beth(x) \) is a stochastic process representing noise (e.g., Gaussian noise with mean 0 and variance \( \tau^2 \)). The \textbf{PSK-operators} is then defined as:
\begin{equation}\label{PSK}
    \mathfrak{P}_n^\beth(f)(x) := \mathfrak{s}_n(f_\beth)(x).
\end{equation}
 
The outline of the article includes the main discussion of SK operators \( \mathfrak{S}_n(f)(x) \), which is used to approximate a function \( f \in L^1(\mathbb{R}^q) \) using convolution with integral mean and sampling kernel. In Section $\eqref{Sec2}$, novel SK operators in the probabilistic framework is discovered called `probabilistic SK-operators (PSK-operators) along with the derivation of a fundamental theorem of approximation with respect to the defined operators. Eventually, we model the real world problem in the form of randomness (image) and investigate key features using PSNR, SSIM, and MAE using the convergence properties in \eqref{M0}
\section{Preliminaries}\label{M0}
\subsection{Basic Classical definition and Matrices}\label{Sec2}
\subsubsection{\textbf{Classical SK-operators}}
Let $ f: \mathbb{R}^q \rightarrow \mathbb{R}$ be a locally integrable function. The classical SK- operators are defined as \eqref{DSK},
where \( \xi \) is a kernel that satisfies standard conditions and \( D_a^n \) is the cell:
\[
D_a^n = \prod_{j=1}^{q} \left[\frac{a_j}{n}, \frac{a_j + 1}{n}\right),
\]
where \( \xi: \mathbb{R}^q \to \mathbb{R} \) is a kernel function satisfying the following properties:
\begin{itemize}
    \item[(i)] \textbf{Unit summation property:} \( \sum_{k \in \mathbb{Z}^q} \xi(x - k) = 1 \) for all \( x \in \mathbb{R}^d \),
    \item[(ii)] \textbf{Kernel Boundedness:} \( \sup_{x \in \mathbb{R}^q} |\xi(x)| < \infty \),
    \item[(iii)] \textbf{Kernel estimation:} There exists \( \delta > 0 \) such that \( |\xi(x)| \leq L (1 + \|x\|)^{-q - \delta} \) for some constant \( L > 0 \).
\end{itemize}
\subsubsection{\textbf{SK-operators in terms of an image}}
    
An image is a representation of the matrix in which each entry denotes the pixel values, but to define the operators in terms of an image using sampling kantorovich operators a 2-dimensional digital grayscale image can be represented
using a step function $ I $ belonging to $L_p(R^2), 1 \leq p < +\infty$, $I$ is defined by:

$$I{(t_1,t_2)}=\sum_{i=1}^{m}\sum_{i=1}^{n} b_{ij}\times \bold{1}_{ij}$$
In the above expression $\bold{1}$ denote the characteristic function such that if $(t_1, t_2)\in (i,i-1]\times (j,j-1]$ then $\bold{1}=1$ otherwise its value will be $0$.\par
\subsection{Arbitrary setting}

Let \( f : \mathbb{R}^q \to \mathbb{R} \) be a measurable function, and let \( \mathfrak{S}_n(f) \) be its approximation using the SK- operators. Then, let us define the following parameters as ;

\subsubsection{\textbf{Point-wise error}} If  $\varsigma_n(x)$ denotes the error between $ \mathfrak{S}_n(f)(x) $ and $f(x) $. Thus, the point error is defined as follows;
\[
\varsigma_n(x) = \left|\mathfrak{S}_n(f)(x) - f(x) \right|.
\]
Similarly, some error metrics in a domain \( D \subset \mathbb{R}^q\)
\begin{itemize}
    \item \textbf{Maximum Error}:
    \[
    \max_{x \in D} \left| \mathfrak{S}_n(f)(x) - f(x) \right|
    \]

    \item \textbf{Minimum Error}:
    \[
    \min_{x \in D} \left| \mathfrak{S}_n(f)(x) - f(x) \right|
    \]

    \item \textbf{Mean \( L^1 \) Error}:
    \[
    \frac{1}{|D|} \int_D \left| \mathfrak{S}_n(f)(x) - f(x) \right| \, dx
    \]

    \item \textbf{Discrete Approximation}:
    \[
    \frac{1}{N} \sum_{i=1}^N \left| \mathfrak{S}_n(f)(x_i) - f(x_i) \right|
    \]
\end{itemize}

\subsubsection{\textbf{Basic classical matrices in image form}}
Let \( f(i,j) \in [0,1] \) represent a grayscale image and \( S_n(f)(i,j) \) its deterministic or stochastic approximation.
\begin{itemize}
    \item \textbf{Max Error}:
    \[
    \max_{(i,j)\in (\mathbb{N}\times \mathbb{N})} \left| \mathfrak{S}_n(f)(i,j) - f(i,j) \right|
    \]
    \item \textbf{Mean Error}:
    \[
    \frac{1}{MN} \sum_{i=0}^{M-1} \sum_{j=0}^{N-1} \left| \mathfrak{S}_n(f)(i,j) - f(i,j) \right|
    \]
    \item \textbf{Structural Similarity Index Measure (SSIM)} 
\[
\text{SSIM}(f, \hat{f}) = \frac{(2\nu_f \nu_{\hat{f}} + C_1)(2\tau_{f\hat{f}} + C_2)}{(\nu_f^2 + \nu_{\hat{f}}^2 + C_1)(\tau_f^2 + \tau_{\hat{f}}^2 + C_2)},
\]
where \( \nu \), \( \tau \), and \( \tau_{f\hat{f}} \) are local means, standard deviations, and cross-covariance.
\item \textbf{Classical Definition:}
\[
\text{MAE}(f, \hat{f}) = \frac{1}{mn} \sum_{i=1}^m \sum_{j=1}^n \left| f_{i,j} - \hat{f}_{i,j} \right|.
\]
\item \textbf{ Peak Signal-to-Noise Ratio (PSNR):}
\[
\text{PSNR}(f, \hat{f}) = 10 \cdot \log_{10} \left( \frac{L^2}{\text{MSE}(f, \hat{f})} \right),
\]
where \( L \) is the maximum possible pixel value (e.g. \( L = 1 \) for normalized images), and
\[
\text{MSE}(f, \hat{f}) = \frac{1}{mn} \sum_{i=1}^m \sum_{j=1}^n (f_{i,j} - \hat{f}_{i,j})^2.
\]

\end{itemize}%

\subsection{Basic Probabilistic definition and Matrices }\label{M1}
\subsubsection{\textbf{Probabilistic Approach for arbitrary function}}
\begin{itemize}
\item \textbf{Notation and Preliminaries}
Let \( \mathfrak{D}\subset \mathbb{R}^n \) denote the compact and measurable domain with Lebesgue measurable with the function \(f\in  L^p(\mathfrak{D}) \) such that \( f: \mathfrak{D} \rightarrow \mathbb{R} \) ,  \( p \geq 1 \). Furthermore, if $\beth $ is a sample space such that for $\omega \in \beth $, let us have a random estimator \( \widehat{f}_n^\beth: \mathfrak{D} \rightarrow \mathbb{R} \) of a function $f,\, n\in \mathbb{N}$. In this paper, the estimator \( \widehat{f}_n^\beth \) is considered to be measurable in both \( x \in \mathfrak{D} \) and \( \omega \in \beth \), and for \( p \geq 1 \), it satisfies \( \widehat{f}_n^\beth \in L^p(\mathfrak{D} \times \beth) \). More preciously, the space $(\beth, \mathfrak{C}\mathfrak{P})$ denotes the probability space, where $\mathfrak{P}$ represents the probability measure in $\mathfrak{C}\subset \Omega$ which refers to a $\tau-$ algebra of events. Let $\mathfrak{E}$ be an expectation operator that corresponds to the probability measure. Mathematically, it can be expressed as follows;
\[
    \mathfrak{E}[g(\beth)] = \int_{\beth} g(\beth) \, d\mathbb{P}(\beth).
\]

    \item As in \cite{Def}, if $f(x)$ and $\widehat{f}_n^\beth(x)$ are the classical arbitrary and stochastic functions corresponding to the arbitrary function $f(x)\,$$\forall x \in \mathfrak{D}$. \( \operatorname{Var}(\cdot) \) and \( \operatorname{Cov}(\cdot, \cdot) \) denote the variance and covariance, receptively. Then for any fixed $x\in \mathfrak{D}$, let us denote  $\nu_f, 
        \nu_{\widehat{f}_n^\beth}, 
        \tau_f^2, \tau_{\widehat{f}_n^\beth}^2,$ and $ \tau_{f, \widehat{f}_n^\beth}$ the  as variances and covariances with respect to the $f(x)$ and $\widehat{f}_n^\beth(x)$, we have the following quantities as follows;


    \begin{align*}
        \nu_f &= \mathfrak{E}[f(x)], &
        \nu_{\widehat{f}_n^\beth} &= \mathfrak{E}[\widehat{f}_n^\beth(x)], \\
        \tau_f^2 &= \operatorname{Var}(f(x)), &
        \tau_{\widehat{f}_n^\beth}^2 &= \operatorname{Var}(\widehat{f}_n^\beth(x)), \\
        \tau_{f, \widehat{f}_n^\beth} &= \operatorname{Cov}(f(x), \widehat{f}_n^\beth(x)).
    \end{align*}

\end{itemize}
\begin{itemize}

    \item \textbf{Expected Mean Absolute Error (MAE)}:
    \[
    \mathfrak{E}\left[ \frac{1}{|\mathfrak{D}|} \int_{\mathfrak{D}} \left| \widehat{f}_n^\beth(x) - f(x) \right| dx \right],
    \]
    where \( \mathfrak{D} \) is the domain over which \( f \) and \( \widehat{f}_n^\beth \) are defined (e.g., \( [0,1]^q \)), and the expectation is over the randomness in the estimator.

    \item \textbf{Expected PSNR (Peak Signal-to-Noise Ratio)}:
    \[
    \mathfrak{E} \left[ 10 \log_{10} \left( \frac{M^2}{\mathfrak{E}[ ( \widehat{f}_n^\beth - f )^2 ]} \right) \right],
    \]
    where \( M \) is the maximum possible value (range) of the signal.

    \item \textbf{Expected SSIM (Structural Similarity Index)}:
    \[
    \mathfrak{E} \left[ \frac{(2 \nu_f \nu_{\widehat{f}_n^\beth} + C_1)(2 \tau_{f, \widehat{f}_n^\beth} + C_2)}{(\nu_f^2 + \nu_{\widehat{f}_n^\beth}^2 + C_1)(\tau_f^2 + \tau_{\widehat{f}_n^\beth}^2 + C_2)} \right],
    \]


    \item \textbf{Variance of absolute error}:
    \[
    \operatorname{Var}(|\widehat{f}_n^\beth - f|) = \mathfrak{E}\left[ |\widehat{f}_n^\beth - f|^2 \right] - \left( \mathfrak{E}\left[ |\widehat{f}_n^\beth - f| \right] \right)^2.
    \]
\end{itemize}

\subsubsection{\textbf{Probabilistic image error metrics}}
Again, let us assume  \( f(i,j) \in [0,1] \) represent a grayscale image and let \( f^\beth(x) \) be a random function (e.g. with additive noise) with $\mathfrak{P}_n^\beth(f)(i,j)$ be its reconstructed stochastic SK operators in terms of a grayscale image. Now, by these assumptions, we will define the following terms as follows:
\begin{itemize}
    \item \textbf{Expected Pixel-wise Error}:
    \[
    \mathfrak{E} \left[ \left| \mathfrak{P}_n^\beth(f)(i,j) - f(i,j) \right| \right]
    \]

    \item \textbf{Expected Mean Pixel-wise Error (over image)}: 
    \[
    \frac{1}{MN} \sum_{i=0}^{M-1} \sum_{j=0}^{N-1} \mathfrak{E} \left[ \left| \mathfrak{P}_n^\beth(f)(i,j) - f(i,j) \right| \right]
    \]
\end{itemize}

\begin{itemize}
    \item  \textbf{Mean Absolute Error (MAE):} The term MAE refers to the absolute measure between the reconstructed image and the original image, across the spatial domain \( \beth \subset \mathbb{R}^d \). In the context of an image, if  \( |\beth| \) denotes the cordiality of an image (image domain) and further, assume that  if \( \mathfrak{P}_n^\beth(f) \) is the reconstructed image corresponding to the original image $f$ , then  the following mathematical expression is given by the following expression as
\begin{equation*}
    \mathfrak{E}\left[\operatorname{MAE}(f, \mathfrak{P}_n^\beth(f))\right] = \mathfrak{E}\left[\frac{1}{|\beth|} \int_{\beth} \left|\mathfrak{P}_n^\beth(f)(x) - f(x)\right| \, dx\right],
\end{equation*}
where, \( \mathfrak{E}[\cdot] \) is expectation over all realizations of the noise.
      
\vskip0.15in
    \item \textbf{Expected Peak Signal-to-Noise Ratio (PSNR):}  The definition of classical PSNR is quit similar to the stochastic PSNR A logarithmic metric that quantifies the ratio between the maximum possible signal power (e.g., intensity range) and the power of the distortion (mean squared error). The expected PSNR is:
    \[
    \mathfrak{E}[\text{PSNR}] = \mathfrak{E} \left[ 10 \log_{10} \left( \frac{M^2}{\text{MSE}(f, \mathfrak{P}_n^\beth(f))} \right) \right],
    \]
    where \( M \) is the dynamic intensity range of the pixel values (e.g., \( M = 1 \) if normalized).
    \item \textbf{Expected SSIM (Structural Similarity Index):} SSIM measures image similarity in terms of structure, luminance, and contrast. Furthermore, assume that \( \nu_f \), \( \nu_{S_n} \) denote the local means of the original and reconstructed images, and \( \tau_f^2 \), \( \tau_{S_n}^2 \) (\( \tau_{f, S_n} \)) represent the local variances (local covariance) of the original and reconstructed images.  Then, the probabilistic expectation is:
    \[
    \mathfrak{E}[\text{SSIM}] = \mathfrak{E} \left[ \frac{(2 \nu_f \nu_{S_n} + C_1)(2 \tau_{f, S_n} + C_2)}{(\nu_f^2 + \nu_{S_n}^2 + C_1)(\tau_f^2 + \tau_{S_n}^2 + C_2)} \right],
    \]
    where, \( C_1, C_2 \) known as a stabilizer constants which is helpful to stabilize the division.
\vskip0.15in
    \item \textbf{Variance of Error:}  As the expression \( |\mathfrak{P}_n^\beth(f)(x) - f(x)| \) is considered as the point-wise absolute error.  The variance of error is defined as:
    \[
    \text{Var}(|\mathfrak{P}_n^\beth(f) - f|) = \mathfrak{E}\left[ |\mathfrak{P}_n^\beth(f) - f|^2 \right] - \left(\mathfrak{E}\left[ |\mathfrak{P}_n^\beth(f) - f| \right]\right)^2.
    \]
    \end{itemize}

\vskip0.15in



%
%


\section{Main results and Algorithms}
\subsection{Lemmas and Theorems:}
\begin{theorem} \label{TH1}
    
Let \( f \in L^1(\mathbb{R}^q) \), and let \( \mathfrak{S}_n(f) \) be the classical SK operators defined as \eqref{DSK}.

Then, the following convergence holds:
\[
\lim_{n \to \infty} \| \mathfrak{S}_n(f) - f \|_{L^1(\mathbb{R}^d)} = 0.
\]
\end{theorem}
For the proof of the above theorem, the reader is requested to go through the article\cite{TH1}.
\begin{lemma}\label{lemma}
Let \( g \in L^1(\mathbb{R}^q) \). Then for SK operators \( \mathfrak{S}_n(g) \), there exists a constant \( L > 0 \), depending only on the kernel \( \xi \), such that
\[
\| \mathfrak{S}_n(g) \|_{L^1} \leq L \| g \|_{L^1},
\]
for all \( n \geq 1 \).
\end{lemma}

\begin{proof}
The classical SK operators are defined by
\[
\mathfrak{S}_n(g)(x) = n^q \sum_{k \in \mathbb{Z}^q} \xi(n x - k) \int_{R_a^n} g(u) \, du,
\]
where \[
D_a^n = \prod_{j=1}^{q} \left[\frac{a_j}{n}, \frac{a_j + 1}{n}\right),
\].

Imposing the \( L^1 \)-norm, we obtain
\begin{eqnarray*}
    \| \mathfrak{S}_n(g) \|_{L^1} &=& \int_{\mathbb{R}^q} \left| n^q \sum_{k \in \mathbb{Z}^q} \xi(n x - k) \int_{D_a^n} g(u) \, du \right| \, dx \\
    &\leq & n^q \sum_{k \in \mathbb{Z}^q} \left( \int_{\mathbb{R}^q} |\xi(n x - k)| \, dx \right) \left( \int_{R_a^n} |g(u)| \, du \right).
\end{eqnarray*}

Assume \( w = n x - k \)  \(\iff \, dw = n^q \, dx \). Thus we have
\[
\int_{\mathbb{R}^q} |\xi(n x - k)| \, dx = \frac{1}{n^d} \int_{\mathbb{R}^q} |\xi(w)| \, dw = \frac{1}{n^w} \|\xi\|_{L^1}.
\]

Substituting, we obtain
\[
\| \mathfrak{S}_n(g) \|_{L^1} \leq \|\xi\|_{L^1} \sum_{k \in \mathbb{Z}^q} \int_{R_k^n} |g(u)| \, du.
\]

Since \( \{R_k^n\} \) forms a partition of \( \mathbb{R}^q \), we have
\[
\sum_{k \in \mathbb{Z}^q} \int_{R_k^n} |g(u)| \, du = \int_{\mathbb{R}^q} |g(u)| \, du.
\]

Thus,
\[
\| \mathfrak{S}_n(g) \|_{L^1} \leq \|\xi\|_{L^1} \| g \|_{L^1},
\]
 setting \( L = \|\xi\|_{L^1} \), thus we have the required results.
\end{proof}
\begin{lemma}
Let \( f \in L^1(\mathbb{R}^q) \) and further let \( \varepsilon_\beth \) be a random perturbation satisfying
\[
\sup_\beth \int_{\mathbb{R}^q} |\varepsilon_\beth(x)| \, dx < \infty.
\]
The PSK-operators are defined as in \eqref{PSK}. 
There exists a constant \( L > 0 \) (depending on the kernel) \( \xi \), such that
\[
\mathfrak{E}\left[ \| \mathfrak{P}_n^\beth(f) \|_{L^1} \right] \leq L \left( \| f \|_{L^1} + \mathfrak{E}\left[ \| \varepsilon_\beth \|_{L^1} \right] \right),
\]
for all \( n \geq 1 \).
\end{lemma}

\begin{proof} Obviously from \eqref{PSK}, the relation between $f(x)$  and perturbation $ \varepsilon_\beth(x)$ is written as follows;
\[
\quad f_\beth(x) = f(x) + \varepsilon_\beth(x),
\]
First, observe that by linearity of \( \mathfrak{S}_n \), we have
\[
\mathfrak{P}_n^\beth(f)(x) = \mathfrak{S}_n(f)(x) + \mathfrak{S}_n(\varepsilon_\beth)(x).
\]

Again, imposing the triangle inequality with respect to \( L^1 \),
\[
\| \mathfrak{P}_n^\beth(f) \|_{L^1} \leq \| \mathfrak{S}_n(f) \|_{L^1} + \| \mathfrak{S}_n(\varepsilon_\beth) \|_{L^1}.
\]

From the previous lemma \eqref{lemma}, we obtain
\[
\| \mathfrak{S}_n(f) \|_{L^1} \leq  L\| f \|_{L^1},
\quad \text{and} \quad
\| \mathfrak{S}_n(\varepsilon_\beth) \|_{L^1} \leq L \| \varepsilon_\beth \|_{L^1},
\]
where \( L = \|\xi\|_{L^1} \) is not dependent on \( n \) and \( \beth \).

Thus,
\[
\| \mathfrak{P}_n^\beth(f) \|_{L^1} \leq L \left( \| f \|_{L^1} + \| \varepsilon_\beth \|_{L^1} \right).
\]

Eventually, operating the expectations on both sides, we have
\[
\mathfrak{E}\left[ \| \mathfrak{P}_n^\beth(f) \|_{L^1} \right] \leq L \left( \| f \|_{L^1} + \mathfrak{E}\left[ \| \varepsilon_\beth \|_{L^1} \right] \right).
\]

By assumption \( \sup_\beth \| \varepsilon_\beth \|_{L^1} < \infty \), the result proved.
\end{proof}

\begin{theorem} \textbf{(Fundamental theorem of approximation)}
    Suppose \( f \in L^1(\mathbb{R}^q) \) and further let us denote the noise process by \( \varepsilon_\beth \) satisfies;
\[
\sup_\beth \int_{\mathbb{R}^d} |\varepsilon_\beth(x)| \, dx < \infty.
\]
Then, from \eqref{PSK} and under the assumption of uniform integrability, we have
\[
\mathfrak{E}\left[ \| \mathfrak{P}_n^\beth(f) - f \|_{L^1} \right] \to 0 \quad \text{as } n \to \infty.
\]
\end{theorem}
\begin{proof}
 From \eqref{PSK}, we have
\[
\mathfrak{P}_n^\beth(f)(x) = \mathfrak{S}_n(f)(x) + \mathfrak{S}_n(\varepsilon_\beth)(x).
\]

Using linearity property over the expectation operators corresponding to \( L^1 \), we have the following estimation:
\[
\mathfrak{E}\left[ \| \mathfrak{P}_n^\beth(f) - f \|_{L^1} \right]
\leq \| \mathfrak{S}_n(f) - f \|_{L^1} + \mathfrak{E}\left[ \| \mathfrak{S}_n(\varepsilon_\beth) \|_{L^1} \right].
\]

From the classical fundamental theorem with respect to \eqref{TH1} for \( \mathfrak{S}_n(f) \), we know:
\[
\| \mathfrak{S}_n(f) - f \|_{L^1} \to 0 \quad \text{as } n \to \infty.
\] Since \( \varepsilon_\beth \in L^1 \) is uniformly bounded in the expectation sense:
\[
\| \mathfrak{S}_n(\varepsilon_\beth) \|_{L^1} \leq L \int_{\mathbb{R}^q} |\varepsilon_\beth(u)| \, du,
\]
for \( L > 0 \) not dependent on \( \beth \) and \( n \).

Therefore,
\[
\mathfrak{E}\left[ \| \mathfrak{P}_n(\varepsilon_\beth) \|_{L^1} \right] \leq L \, \mathfrak{E}\left[ \int_{\mathbb{R}^q} |\varepsilon_\beth(u)| \, du \right].
\]

By assumption,
\[
\sup_\beth \int_{\mathbb{R}^q} |\varepsilon_\beth(u)| \, du < \infty \quad \Rightarrow \quad \mathfrak{E}\left[ \int_{\mathbb{R}^q} |\varepsilon_\beth(u)| \, du \right] < \infty.
\]
Eventually, by the assumption of uniform integrability and possibly additional assumptions (particularly \( \varepsilon_\beth \to 0 \) in \( L^1 \)), we can conclude:
\[
\lim_{n \to \infty} \mathfrak{E}\left[ \| \mathfrak{P}_n(\varepsilon_\beth) \|_{L^1} \right] = 0.
\]
Thus, we have
\[
\mathfrak{E}\left[ \| \mathfrak{P}_n^\beth(f) - f \|_{L^1} \right] \to 0 \quad \text{as } n \to \infty.
\]

\end{proof}

\subsection{Algorithm 1: Approximating $f(x) = e^{-x^2}$ using Classical and Probabilistic SK-operators} \label{algo}
\vskip 0.15in
\textbf{Input:} Function $f(x) = e^{-x^2}$, domain $x \in [-3, 3]$, number of points $=1000$, $n$-values $= \{5, 15, 25, 35, 45\}$ \\
\textbf{Output:} Approximated functions $\mathfrak{S}_n f(x)$ for Classical and Probabilistic SK operators

\vskip 0.15in
\textbf{Define:} $f(x) = e^{-x^2}$, $x = \text{linspace}(-3, 3, 1000)$

\vskip 0.1in
\textbf{Kernel Function:} For  SK, use the indicator kernel
\[
\xi_{[x_k, x_k + h)}(x) = 
\begin{cases}
1, & \text{if } x \in [x_k, x_k + h) \\
0, & \text{otherwise}
\end{cases}
\]

\vskip 0.1in
\textbf{For each} $n \in \{5, 15, 25, 35, 45\}$ \textbf{do:}
\vskip 0.15in
\hrule
\vskip 0.15in
\textbf{Classical SK Operators:}
\begin{itemize}
    \item Set $h = \frac{1}{n}$
    \item Initialize $\mathfrak{S}_n f(x) = 0$ over all $x$
    \item \textbf{For} $k = -3n$ \textbf{to} $3n$:
    \begin{itemize}
        \item Compute $x_k = k \cdot h$
        \item Compute ${Kr}_k = \frac{1}{h} \int_{x_k}^{x_k + h} f(u)\,du$
        \item \textbf{For each} $x_i$:
        \begin{itemize}
            \item \textbf{If} $x_i \in [x_k, x_k + h)$ \textbf{then} set $\mathfrak{S}_n f(x_i) = Kr_k$
        \end{itemize}
    \end{itemize}
\end{itemize}
\vskip 0.15in
\hrule
\vskip 0.15in
\vskip 0.15in
  
\textbf{Probabilistic SK Operators:}
\begin{itemize}
    \item Set $h = \frac{1}{n}$
    \item Initialize $\mathfrak{P}_n f(x) = 0$ over all $x$
    \item Initialize Gaussian noise generator ($\varepsilon_\beth$)  (mean = 0, std = 0.02)
    \item \textbf{For} $k = -3n$ \textbf{to} $3n$:
    \begin{itemize}
        \item Compute $x_k = k \cdot h$
        \item Compute $Kr_k = \frac{1}{h} \int_{x_k}^{x_k + h} f(u)\,du$
        \item Sample perturbation $\delta \sim \mathcal{N}(0, 0.02^2)$
        \item \textbf{For each} $x_i$:
        \begin{itemize}
            \item \textbf{If} $x_i \in [x_k, x_k + h)$ \textbf{then} set $\mathfrak{P}_n f(x_i) = Kr_k + \delta$
        \end{itemize}
    \end{itemize}
\end{itemize}

\vskip 0.15in
\hrule
\vskip 0.15in
\textbf{Plotting:}
\begin{itemize}
    \item Plot original function $f(x) = e^{-x^2}$
    \item \textbf{For each} $n$:
    \begin{itemize}
        \item Plot Classical SK: $\mathfrak{S}_n f(x)$
        \item Plot Probabilistic SK: $\mathfrak{P}_n f(x)$ with noise
    \end{itemize}
\end{itemize}
\vskip 0.15in
\hrule
\vskip 0.15in
\section{Numerical Illustration}
\begin{example}
    We consider the function:
\[
f(x) = e^{-x^2}, \quad x \in \mathbb{R}
\]


The classical SK-operators in one dimension which is defined as:
\[
\mathfrak{S}_n(f, x) = \sum_{k \in \mathbb{Z}} \left( n \int_{\frac{k}{n}}^{\frac{k+1}{n}} f(u) \, du \right) \cdot \varphi(n x - k)
\]
where:\\
Let \( f: \mathbb{R} \to \mathbb{R} \) be a locally integrable function, and \( n \in \mathbb{N} \) a sampling density. Define the subintervals

\[
I_k = \left[ \frac{k}{n}, \frac{k+1}{n} \right), \quad \text{for } k \in \mathbb{Z}.
\]
and \( \varphi \) is a suitable kernel function. Particularly, here \[
\varphi_{I_k}(x) = 
\begin{cases}
1, & \text{if } x \in I_k, \\
0, & \text{otherwise}.
\end{cases}
\]
where \( \varphi_{I_k}(x) \) is the characteristic function of the interval \( I_k \).



\vskip.15in


\bigskip

The PSK operators Kantorovich operators with perturbation introduces randomness into the average as:
\[
\mathfrak{P}_n(f, x) = \sum_{k \in \mathbb{Z}} \left( n \int_{I_k} f(u)\,du + \varepsilon_k \right) \cdot \varphi_{I_k}(x),
\]
where particularity, \( \varepsilon_k \) is a random variable of real value as \( \varepsilon_k \sim \mathcal{N}(0, \tau^2) \) representing noise added to the local average over each interval \( I_k \).
\vskip.15in
Now, we give the numerical data in the following table, where we compare the classical and probabilistic SK- operators at sample level as;
\begin{table}[ht]
\centering
\begin{tabular}{|c|c|c|}
\hline
\textbf{Sampling Level} \( n \) & \textbf{Classical SK Error} & \textbf{Probabilistic SK Error} \\
\hline
5  & 0.086 & 0.091 \\
15 & 0.043 & 0.048 \\
25 & 0.021 & 0.027 \\
35 & 0.010 & 0.018 \\
\hline
\end{tabular}
\vskip.15in
\caption{Comparison of $L^1$-errors between Classical and Probabilistic SK operators for various sampling levels \( n \).}
\end{table}

\begin{figure}
    \centering
   \includegraphics[width=1\textwidth]{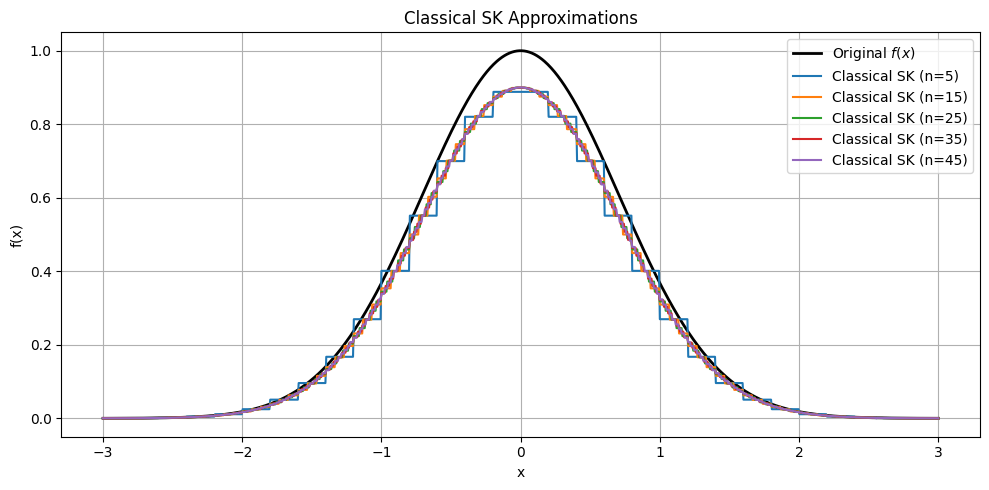}
     \caption{Approximation of \( f(x) = e^{-x^2} \) using classical SK-Operator for various \( n \).}
    \label{fig:fejer_results}
\end{figure}
\begin{figure}
    \centering
    \includegraphics[width=1\textwidth]{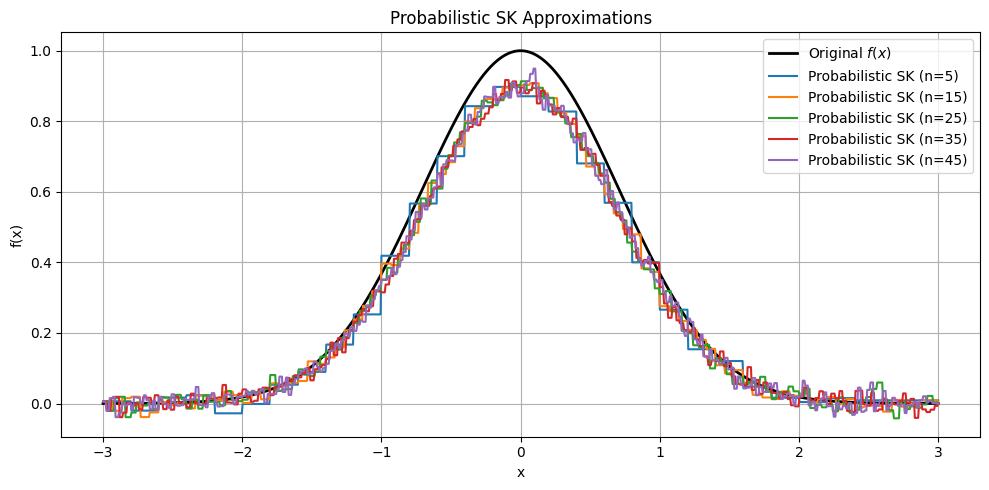}
        \caption{Approximation of \( f(x) = e^{-x^2} \) using PSK Operators for different \( n \).}
    \label{fig:fejer_results}
\end{figure}
\vskip0.15in
From the above table, we conclude that classical sampling operators work well whereas the PSK-sampling operators are bit less efficient. Both the operators are implemented under the ideal condition (no randomness/noise). Preciously, the classical convergence is faster than the convergence of classical SK-operators under $L^1(\mathbb{R})$.\par
\end{example}
\newpage

In the real-life situation problems, we construct the behavior of randomness or noisy image (non-ideal situation) for both the operators.
 In the following problem, let us do a detailed analysis;

\begin{example} The operators in classical and stochastic operators defined in equations \eqref{DSK} and \eqref{PSK}, respectively, and a discrete function $f$ in the form of image, specifically the widely-used cameraman test image is used to approximate. The scheme of the operators is implemented over Python according to the algorithm \eqref{algo}.\par
\vskip.15in
We investigated the various aspects of the image in the forms of mathematical parameters such as PSNR, SSIM, and MAE at various levels of samples. In this work, particularly in the context of image, we construct the $ n\times n $ window using sample $n$ with the SK operators as shown in Table 3.\par
\vskip.15in
On the other hand, while dealing with the PSK operators, we find out the approximation parameters like PSNR, SSIM, and MAE with in  expectation operators form. Similarly, we perform the algorithm \eqref{algo} for PSK-operators to obtain the nature of images at various levels.
\vskip.15in
    \begin{itemize}
  \item \textbf{Mathematical parameters with respect to the classical SK-operators }
  \end{itemize}
  \vskip.15in
    As shown in the figure $3$, it is clearly shown the cameraman photo at various windows. The images gets more smooth as we increase the window size and from the practical point of view error refers the detection of image at a certain window which can be seen in Figure $3$
\vskip.15in
The following table shows the data corresponding to the SK- operators as; 
\begin{table}[h!]
\centering
\begin{tabular}{|c|c|c|c|c|}
\hline
\textbf{Window Size} & $\text{PSNR}$ & $\text{SSIM}$ & $\text{MAE}$ & $\text{Var}(|\mathfrak{S}_n(f)-f|)$ \\
\hline
$3 \times 3$   & 29.45 & 0.8590 & 0.0174 & 0.00083128 \\
$7 \times 7$   & 25.10 & 0.7185 & 0.0286 & 0.00227373 \\
$15 \times 15$ & 22.21 & 0.6155 & 0.0402 & 0.00439862 \\
\hline
\end{tabular}
\vskip0.15in
\caption{Quantitative evaluation of SK approximation with different window sizes.}
\end{table}

\newpage

   \begin{figure}
		\includegraphics[width=1.1\linewidth]{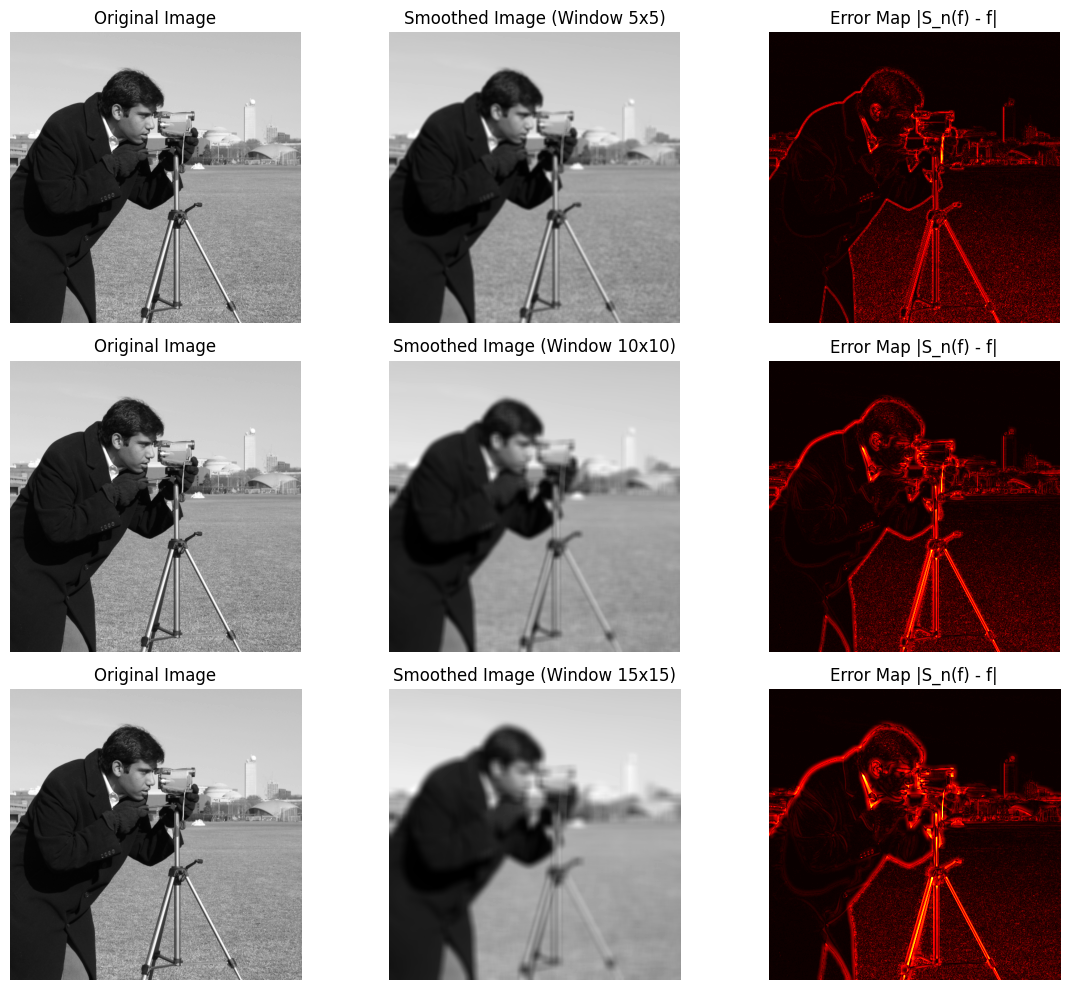}
		\caption{SK-operators}
        
	\end{figure}
    \begin{figure}
		\includegraphics[width=1.1\linewidth]{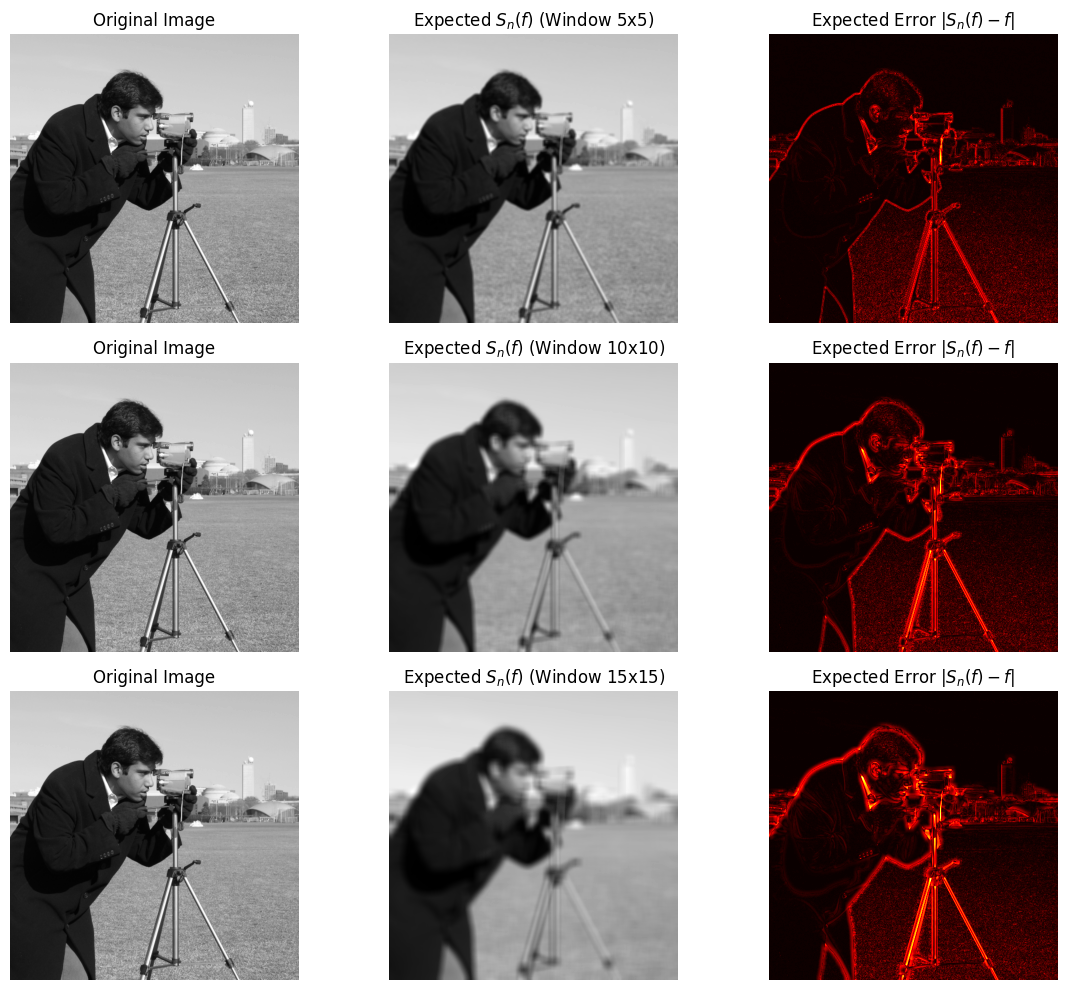}
		\caption{SK-operators in probabilistic way}
	\end{figure}

\begin{itemize}
\newpage
 \item \textbf{Mathematical parameters with respect to the PSK-operators}

    \begin{table}[h!]
\centering
\begin{tabular}{|c|c|c|c|c|}
\hline

\textbf{Window Size} & $\mathfrak{E}[\text{PSNR}]$ & $\mathfrak{E}[\text{SSIM}]$ & $\mathfrak{E}[\text{MAE}]$ & $\text{Var}(|\mathfrak{S}_n(f)-f|)$ \\
\hline
$3 \times 3$   & 28.48 & 0.8421 & 0.0305 & 0.000129 \\
$7 \times 7$   & 28.98 & 0.8797 & 0.0274 & 0.000081 \\
$15 \times 15$ & 27.86 & 0.8446 & 0.0283 & 0.000047 \\
\hline
\end{tabular}
\caption{Quality metrics (Expected PSNR, SSIM, MAE, and Variance) for expected \( S_n(f) \) on a grayscale image with different window sizes.}
\end{table}
\vskip.15in
\vskip.15in
\vskip.15in
\vskip.15in
\vskip.15in
\vskip.15in
\vskip.15in
\vskip.15in
\vskip.15in
\vskip.15in
\vskip.15in
\vskip.15in

\newpage

\begin{table}[h!]
\centering
\item \textbf{Tabulated Comparison of Application Insights} \\[0.3cm]
\begin{tabular}{|p{3.5cm}|p{5.5cm}|p{5.5cm}|}
\hline
\textbf{Key features} & \textbf{Classical Interpretation} & \textbf{Probabilistic Interpretation} \\
\hline
\textbf{Nature of $f(x)$} & Image \( f(x) \) & Noisy version \( f^\beth(x) = f(x) + \eta(x,\beth) \) \\
\hline
\textbf{Operators Type} & Standard classical SK-operators \( \mathfrak{S}_n(f) \) & Noisy PSK-operators \( \mathfrak{P}_n^\beth(f) \) with expectation \\
\hline
\textbf{Stability Nature} & Stable under ordinary conditions, but sensitive to noise & Robust under randomness, improves with group averaging \\
\hline
\textbf{Error Analysis} & Point-wise and \( L^1 \)-error, max and min classical errors & Probabilistic errors: \( \mathfrak{E}[|\mathfrak{P}_n^\beth(f)(x) - f(x)|] \), variances \\
\hline
\textbf{Application in Image Processing} & Denoising, enhancement, interpolation using averaging kernels & Noise reduction, probabilistic restoration and filtering \\
\hline
\textbf{Approximation via Metrics} & 
\begin{minipage}[t]{\linewidth}
\begin{itemize}
    \item PSNR (Resolution, enhancement)
    \item SSIM (Structural based parameters )
    \item MAE
\end{itemize}
\end{minipage}
& 
\begin{minipage}[t]{\linewidth}
\begin{itemize}
    \item Probabilistic PSNR: \( \mathfrak{E}[\text{PSNR}] \)
    \item Probabilistic SSIM: \( \mathfrak{E}[\text{SSIM}] \)
    \item Variance of error for consistency
\end{itemize}
\end{minipage} \\
\hline
\textbf{Convergence Study} & Converges to \( f(x) \) in \( L^p \)-norm or uniformly for smooth \( f \) & Converges in probability, or almost surely under integrability and summability conditions \\
\hline
\end{tabular}
\vskip.15in
\vskip.15in
\caption{Comparison between SK-operators and PSK-operators via standard metrics.}
\end{table}
\end{itemize}

\centering
\vskip.15in
\vskip.15in
\vskip.15in
\vskip.15in
\vskip.15in
\vskip.15in
\vskip.15in
\vskip.15in
\vskip.15in
\vskip.15in
\newpage

\vskip.15in
\vskip.15in

\vskip.15in
\vskip.15in
\vskip.15in
\vskip.15in
\vskip.15in
\vskip.15in
\vskip.15in
\vskip.15in
\vskip.15in
\vskip.15in
\vskip.15in
\vskip.15in
\newpage
\vskip.15in
\vskip.15in

\end{example}
In the above discussion, we have demonstrated the empirical and mathematical study in specific parameters like PSNR, SSIM and MAE of images at a certain level of sample values. Moreover, the error and expected error reveal an important feature called `detection of image' for various samples.
 Consequently, let us give a comparative analysis of various aspects such as stability, error matrices, etc. between the SK-operators and the PSK-operators `in Table 4.

\section{Conclusion} In this research paper, the authors have made an effort to discover novel probability SK operators and have derived the fundamental theorem of approximation. Analogously, a deep comparative study is established between the classical and probabilistic SK operators not only in the form of numerical data but also in the key features of an image in which the classical results are better than the probabilistic results. In contrast, the PSK- sampling operators provide better results as compared to SK-operators when dealing with the noise or randomness. More preciously, if the variance is low, the probabilistic operators yield stable results, which implies that the probabilistic SK-operators yield stable results.
\vspace{0.55cm}
\begin{center}
\textbf{Acknowledgments}
\end{center}
As the first author, I consider this work a significant step in my academic journey and a source of renewed inspiration for future research. It has strengthened my motivation to continue at the postdoctoral level, where I aspire to contribute meaningfully to the advancement of knowledge. I am driven by a vision to join the esteemed class of researchers whose work pushes me to invest my life in research.

Furthermore, he acknowledges the financial support provided by the ``Homi Bhabha Teaching cum Research Fellowship" from Dr. A. P. J. Abdul Kalam Technical University, Lucknow, India.
\vspace{0.35cm}
\begin{center}
\textbf{Compliance with Ethical Standards}
\end{center}
\textbf{Funding:} No funding was received to report this study.\\
\textbf{Conflict of interest:} All authors declare that there is no conflict of interest.\\
\textbf{Ethical approval:} This article does not contain studies with human participants or animals performed by any of the authors.\\
\textbf{Data availability:} We state that no data sets were generated or analyzed during the preparation of the manuscript.\\
\textbf{Code availability:} Not applicable.\\
\textbf{Authors' contributions:} All the authors have equally contributed to the conceptualization, framing, and writing of the manuscript.

\end{document}